\RequirePackage{fix-cm}

\documentclass[smallextended,envcountsect]{svjour3}       

\usepackage{bm}
\usepackage{graphicx}

\usepackage{amsmath,amsfonts,epsfig,amsthm, mathrsfs}
\usepackage[utf8]{inputenc}
\usepackage{hyperref}
\hypersetup{
	colorlinks=true,
	linkcolor=blue,
	filecolor=magenta,      
	urlcolor=cyan,
	citecolor=blue
}
\smartqed

\makeatletter
\@addtoreset{equation}{section}
\renewcommand{\thesection}{\arabic{section}}

\newcommand{\beq}[1]{\begin{equation} \label{#1}}
	\newcommand{\eeq}{\end{equation}}
\newcommand{\bed}{\begin{displaymath}}
	\newcommand{\eed}{\end{displaymath}}
\newcommand{\bea}{\bed\begin{array}{rl}}
	\newcommand{\eea}{\end{array}\eed}
\newcommand{\barray}{\begin{array}{ll}}
	\newcommand{\earray}{\end{array}}

\def\qedd
{\strut\hfill\lower0.5\baselineskip\vbox{\hrule width
		1em\nointerlineskip
		\hbox to 1em{\vrule height 1em
			\hfill
			\vrule height 1em}
		\nointerlineskip
		\hrule width 1em}\bigbreak}
\def\qedsym
{\vbox{\hrule width 0.5em\nointerlineskip
		\hbox to 0.5em{\vrule height 0.5em
			\hfill
			\vrule height 0.5em}
		\nointerlineskip
		\hrule width 0.5em}}

\spnewtheorem{thm}{Theorem}[section]{\bfseries}{\itshape}
\spnewtheorem{prop}[thm]{Proposition}{\bfseries}{\itshape}
\spnewtheorem{lem}[thm]{Lemma}{\bfseries}{\itshape}
\spnewtheorem{result}[thm]{Result}{\bfseries}{\itshape}
\spnewtheorem{rem}[thm]{Remark}{\bfseries}{\itshape}
\spnewtheorem{cor}[thm]{Corollary}{\bfseries}{\itshape}
\spnewtheorem{defn}[thm]{Definition}{\bfseries}{\itshape}
\spnewtheorem{exm}[thm]{Example}{\bfseries}{\itshape}

\def\R{\mathbb{R}}

\def\P{\mathbb{P}}
\def\E{\mathbb{E}}
\def\dx{\operatorname{d}}

\begin{document}
	
	\title{Non-uniform bounds for non-normal approximation via Stein's method with applications to the Curie--Weiss model and the imitative monomer-dimer model}

	\titlerunning{Non-uniform bounds for non-normal approximations via Stein's method}        
	
	\author{L\^{e} V\v{a}n Th\`{a}nh$^*$\thanks{$^*$Corresponding author} \and Nguyen Ngoc Tu}
	
	\institute{L\^{e} V\v{a}n Th\`{a}nh \at 
		Department of Mathematics, Vinh University, Nghe An, Vietnam. 
		\email{levt@vinhuni.edu.vn}\\
		Nguyen Ngoc Tu \at
		Department of Applied Sciences, HCMC University of Technology and Education, 01 Vo Van Ngan, Thu Duc District, Ho Chi Minh City, Vietnam.
		\email{tunn@hcmute.edu.vn}}

	\date{Received: date / Accepted: date}
	
	\maketitle
	
	\begin{abstract}
		This paper establishes a non-uniform Berry--Esseen bound for non-normal approximation using Stein's method. 
		The main theorem generalizes the result of the authors in [Comptes Rendus Math\'{e}matique, 2024] to the context of non-normal approximation. 
		As an application of the main result, we derive non-uniform Berry--Esseen bounds in non-central limit theorems for the magnetization in the
		Curie--Weiss model and the imitative monomer-dimer model.
These extend some existing results in the literature, including Theorem 2.1 of Chatterjee and Shao [Ann. Appl. Probab., 2011] and Theorem 1 of Chen [J. Math. Physics, 2016].
		
		
	\end{abstract}
	
	\section{Introduction}
	
	Let $(W, W')$ be an exchangeable pair, that is, $(W,W')$ and $(W',W)$ have the same distribution. 
	Stein \cite{stein1986approximate} developed a new method and proved a normal approximation for $W$ under the linear regression condition
	\begin{equation}\label{linear}
		\E(W'|W) = (1-\lambda) W,
	\end{equation}
	for some $0<\lambda<1$.  
	Chatterjee and Shao \cite{chatterjee2011nonnormal}, and Eichelsbacher and L\"{o}we \cite{eichelsbacher2010stein} expanded Stein's method to accommodate a broader class of distributional approximations beyond the normal case. They generalized the linear regression condition \eqref{linear} to a nonlinear form
	\begin{equation}\label{nonlinear}
		\mathbb{E}\left(W^{\prime} \mid W\right)=  W-\lambda \psi(W) + R,
	\end{equation}
	where $0<\lambda<1$, function $\psi(x)$ depends on a continuous distribution under consideration and $R$ is a random variable.
	
	While much of the work on Stein's method has concentrated on uniform error bounds, 
	several achievements have demonstrated its effectiveness in proving non-uniform Berry--Esseen bounds (Chen et al. \cite{chen2021error}, Chen and Shao \cite{chen2001non,chen2004normal})
	and the Cram\'{e}r-type moderate deviation (Chen et al. \cite{chen2013stein}, Mukherjee et al. \cite{mukherjee2024moderate}, Shao et al. \cite{shao2021cramer}, Zhang \cite{zhang2023cramer}). 
	Very recently, Butzek and Eichelsbacher \cite{butzek2024non}, Dung et al. \cite{dung2025non}, Eichelsbacher \cite{eichelsbacher2024stein}, Liu et al. \cite{liu2021non}, and Th\`{a}nh and Tu \cite{thanh2025non} extended Stein's method to obtain non-uniform bounds in
	normal and non-normal approximations.
Compared to the Cram\'{e}r-type moderate deviation, 
non-uniform Berry--Esseen results typically require weaker conditions and provide a bound with a more explicit constant.
	
	The present paper is a follow-up to \cite{thanh2025non},
	where the authors used a non-uniform concentration inequality approach in Stein's method for normal approximation
	for exchangeable pairs. Here, we generalize the main theorem of \cite{thanh2025non} to cover both normal and non-normal approximations.
	The main result is applied to derive non-uniform Berry--Esseen bounds in non-central limit theorems for the magnetization in the
	critical Curie--Weiss model
and the critical imitative monomer-dimer model. These results extend Theorem 2.1 of Chatterjee and Shao \cite{chatterjee2011nonnormal}, 
Theorem 3.8 of Eichelsbacher and L\"{o}we \cite{eichelsbacher2010stein} and Theorem 1 of Chen \cite{chen2016limit} to the non-uniform Berry--Esseen bounds.

	The rest of the paper is organized as follows.
	Section \ref{sec.main} presents a general theorem on non-uniform Berry--Esseen bound for non-normal approximation of exchangeable pairs.
	In Section \ref{sec.proof}, we present five preliminary lemmas and the proof of Theorem \ref{thm.main11}. 
	Non-uniform Berry--Esseen bounds 
	in the non-central limit theorems for the magnetization in the critical Curie--Weiss model and the
	critical imitative monomer-dimer model
	are presented in Section \ref{sec.appl}.

	\section{A non-uniform Berry--Esseen bound for non-normal approximation of exchangeable pairs}\label{sec.main}
	
	Let $k$ be a positive integer, $a_k$ a positive constant and
	$Y_k$ a random variable with probability  density function
	\begin{equation}\label{b1}
		p_k(x)= b_k \exp\left\{-a_k x^{2k}\right\},\text{ } x\in\R,
	\end{equation}
	where $b_k = \left(\int_{\mathbb{R}} \exp\left\{-a_k x^{2k}\right\}\dx x \right)^{-1}$ is the normalizing constant.
	Note that when $k=1$ and $a_k=1/2$, $Y_k$ is a standard normal distribution random variable.
	Let $(W, W')$ be an exchangeable pair satisfying
	\begin{equation}\label{nonlinear2}
		\mathbb{E}\left(W^{\prime} \mid W\right)=  W-\lambda \psi(W) + R,
	\end{equation}
	where $0<\lambda<1$ and $\psi(x)=2ka_kx^{2k-1}$, $x\in \R$.
Eichelsbacher and L\"{o}we \cite{eichelsbacher2010stein}
	established uniform Berry--Esseen bounds for the random variable $W$. From Lemma 2.2 and Theorem 2.4 of Eichelsbacher and L\"{o}we \cite{eichelsbacher2010stein},
	we have the following theorem.
	
	\begin{theorem}[Eichelsbacher and L\"{o}we \cite{eichelsbacher2010stein}]\label{thm.eichelsbacher2010}
		Let $k$ be a positive integer and let $Y_k$ be a random variable distributed according to $p_k$ as given in \eqref{b1}. 
		Let $(W, W')$ be an exchangeable pair of random variables satisfying \eqref{nonlinear2}, and let $\Delta := W - W'$. 
		Then 
		\begin{equation*}\label{eiche15}		
			\begin{split}
				&\sup_{z\in\R}\left| \mathbb{P}(W \le z) - \mathbb{P}(Y_k \le z) \right|\\
				&\le 
				C(1+\E|W|^{2k-1})\left( \sqrt{\E\left( 1- \frac{1}{2\lambda}\E({\Delta}^2|W)\right)^2 } +\dfrac{\sqrt{\E R^2}}{\lambda} 
				+  a +\frac{a^3}{\lambda}+ \dfrac{\E\Delta^2\mathbf{1}(|\Delta|>a)}{\lambda}\right),
			\end{split}	
		\end{equation*}
		where $C$ is a constant depending only on the density $p_k$.
	\end{theorem}
	
	Recently, Eichelsbacher \cite{eichelsbacher2024stein} used results from Liu et al. \cite{liu2021non} and Shao and Zhang \cite{shao2019berry}
	to obtain a non-uniform Berry--Esseen bound for bounded exchangeable pairs.
The following result is a consequence of Theorem 2.5 of Eichelsbacher \cite{eichelsbacher2024stein}.
	\begin{theorem}[Eichelsbacher \cite{eichelsbacher2024stein}]\label{thm.eichelsbacher2024}
		Let $k$ be a positive integer and let $Y_k$ be a random variable distributed according to $p_k$ as given in \eqref{b1}. 
		Let $(W, W')$ be an exchangeable pair of random variables satisfying \eqref{nonlinear2}, and let $\Delta := W - W'$. 
		Assume that $\E|\psi(W)|^{2} < \infty$ and $|\Delta| \le a$ for some $a>0$. Then for any $z \in \mathbb{R}$, we have
		\begin{equation}\label{eiche25}		\left| \mathbb{P}(W \le z) - \mathbb{P}(Y_k \le z) \right| \le 
			\dfrac{C}{\left( 1 + |z|^{2k-1} \right)}  
			\left(\left| \mathbb{E} \left( 1 - \frac{1}{2\lambda}\mathbb{E}(\Delta^2|W) \right) \right| + \dfrac{\mathbb{E}|R|}{\lambda} + 3a\right),
		\end{equation}
		where $C$ is a constant depending only on $p_k$ and $\E|\psi(W)|^2$.
	\end{theorem}
	
	In Remark 3 of Eichelsbacher \cite{eichelsbacher2024stein}, the author noted that if the random variable $R$ in \eqref{nonlinear2} is absent, 
	then the term $1/(1+|z|^{2k-1})$ in \eqref{eiche25} can be improved to $1/(1+|z|^{p})$ for $p\ge 2k-1$ 
	provided that $\E|W|^{2p}<\infty$.
	However, in many applications, the exchangeable pair $(W,W')$
	only satisfies \eqref{nonlinear2}, making it essential to develop new techniques to address this limitation.
	
	The main objective of the present paper is to generalize the main result in \cite{thanh2025non}
	to cover both normal and non-normal approximations. Our proof is based on a non-uniform concentration inequality, which differs from those used by
	Eichelsbacher \cite{eichelsbacher2024stein} and Liu et al. \cite{liu2021non}.
	This approach allows for strengthening the term
	$1/(1+|z|^{2k-1})$ in \eqref{eiche25} to $1/(1+|z|)^p$ for any $p\ge 2k-1$, even in the presence of the remainder 
	$R$. This advancement broadens the applicability of non-uniform Berry--Esseen bounds and provides sharper error estimates in various models.
	The main result, Theorem \ref{thm.main11}, also extends Theorem \ref{thm.eichelsbacher2024} to the case where $\Delta$ may not be bounded.
	For the normal case ($k=1$ and $a_k=1/2$), Theorem \ref{thm.main11} reduces to the case of non-uniform bound for normal approximation, which is Theorem 1.2 in \cite{thanh2025non}. Moreover, the constant $C(p)$ in Theorem 1.2 in \cite{thanh2025non} is not as explicit as in Theorem \ref{thm.main11}.
	
	\begin{theorem}\label{thm.main11} 
		Let $k$ be a positive integer and let $Y_k$ be a random variable distributed according to $p_k$ as given in \eqref{b1}.
		Let $(W, W')$ be an exchangeable pair of random variables satisfying \eqref{nonlinear2} and let $\Delta=W-W'$. 
		Then for any $a>0$, $p\ge 2k-1$ and $z\in\mathbb{R}$, we have
		\begin{equation*}
			\begin{split}
				&|\P(W\le z) - \P(Y_k\le z)| \\
				&\le\frac{B^p(p^{p/k}+\E |W|^{2p})}{(1+|z|)^{p} } \left(\sqrt{\E\left( 1- \frac{1}{2\lambda}\E({\Delta}^2|W)\right)^2} +\dfrac{\sqrt{\E R^2}}{\lambda} 
				+  a +\frac{ a^3}{\lambda}+ \dfrac{\sqrt{\E\Delta^4\mathbf{1}(|\Delta|>a)}}{\lambda}\right),
			\end{split}
		\end{equation*}
		where $B$ is a constant which depends only on $k$ and $a_k$. 
		
	\end{theorem}
	
	\section{Proof of Theorem \ref{thm.main11}}\label{sec.proof}
	In this section, we provide a proof of Theorem \ref{thm.main11}. Throughout this section, we use the notations defined in Section \ref{sec.main}. 
	The symbol $B$ denotes a constant which depends only on $k$ and $a_k$ and can be different for each appearance. 
	If $\E|W|^{2p}=\infty$, then the result is trivial. Therefore, it suffices to consider the case $\E|W|^{2p}<\infty$.
	
	To establish Theorem \ref{thm.main11}, we first introduce some preliminary lemmas.
	The first lemma is a simple consequence of Lyapunov's inequality. See
	Th\`{a}nh and Tu \cite{thanh2025non} for a proof.
	\begin{lemma}\label{lem.bound.moment.W} 
		Let $p\ge 1$, $0<s\le 2p$ and let $|c|\le 100$. Then
		\begin{equation}\label{estimate.moment.W01}
			\E|W+c|^s\le B^p(1+\E|W|^{p}).
		\end{equation}
	\end{lemma} 
	
	The next lemma generalizes the non-uniform concentration inequality in Lemma 2.2 of 
	Th\`{a}nh and Tu \cite{thanh2025non}, and its proof follows a similar approach to that of 
	Lemma 2.2 of Th\`{a}nh and Tu \cite{thanh2025non}.
	\begin{lemma}\label{lem.concentration.inequality} 
		Let $p\ge 2k-1$ and $z \ge 5$. Then for $0<a< 1$, we have
		\begin{equation} \label{T322a}
			\E \left(\Delta^2\mathbf{1}(|\Delta|\le a)\mathbf{1}(z-a \le W \le z+a)\right) \le \dfrac{B^{p}(1+\E|W|^{2p})a(\sqrt{\E R^2} + \lambda)}{(1+z)^p}.
		\end{equation}
	\end{lemma}

	\begin{proof}
		As in Lemma 2.2 of Th\`{a}nh and Tu \cite{thanh2025non}, we set
		\[f(x) = \begin{cases}
			0 &\text{ for } x < z-2a,  \\
			(1+x+2a)^p(x-z+2a) &\text{ for } z-2a \le x \le z+2a,  \\
			4a(1+x+2a)^p &\text{ for }x > z+2a.  \\ 
		\end{cases} \]
		Applying the exchangeability and \eqref{nonlinear2}, we have
		\begin{equation}\label{sp}
			\begin{split}
				\E\left((W-W')(f(W)-f(W'))\right)&= 2\E Wf(W)-2\E W'f(W)\\
				&= 2\E Wf(W)-2\E(f(W)\E(W'|W))\\
				&= 2\E Wf(W)-2\E f(W)(W - \lambda \psi(W) + R)\\
				&=2\lambda \E \psi(W)f(W)-2\E f(W)R.
			\end{split}
		\end{equation}
		We note that $f$ is continuous, increasing, $f' \ge 0$ on $\mathbb{R}$, and 
		\begin{equation}\label{derivative.of.f}
			\begin{split}
				f'(x) \ge (1+x+2a)^p \ge (1+z)^p \text{ for }z-2a \le x \le z+2a.
			\end{split}
		\end{equation}
		Therefore
		\begin{equation}\label{LHS} 
			\begin{split}
				&\E\left((W-W')(f(W)-f(W'))\right)\\
				&= \E\Delta \int_{-\Delta}^0 f'(W+t)\dx t \\
				&\ge \E\Delta \int_{-\Delta}^0 \mathbf{1}(|t|\le a)\mathbf{1}(z-a\le W \le z+a)f'(W+t)\dx t \\
				&\ge \E\Delta \int_{-\Delta}^0 \mathbf{1}(|t|\le a)\mathbf{1}(z-a \le W \le z+a)(1+z)^p\dx t \\
				&\ge (1+z)^p\E\Delta \int_{-\Delta}^0 \mathbf{1}(|\Delta|\le a)\mathbf{1}(z-a \le W \le z+a)\dx t \\
				&= (1+z)^p \E\left(\Delta^2\mathbf{1}(|\Delta|\le a)\mathbf{1}(z-a \le W \le z+a)\right),
			\end{split}
		\end{equation}
		where we have applied \eqref{derivative.of.f} in the second inequality.
		By using 
		definition of $f$, the Cauchy--Schwarz inequality
		and Lemma \ref{lem.bound.moment.W}, we have
		\begin{equation}\label{RHS} \begin{split} 
				&2\lambda \E \psi(W)f(W)-2\E f(W)R\\
				&\le 2\lambda \E(4a |1+W+2a|^p 2ka_k|W|^{2k-1})+2\E(4a|1+W+2a|^p|R|)\\
				&\le 8a \left(2ka_k\lambda\sqrt{\E|W+1+2a|^{2p}\E |W|^{4k-2}}+\sqrt{\E|W+1+2a|^{2p}}\sqrt{\E R^2}\right) \\
				&\le B^p(1+\E|W|^{2p})a(\lambda+\sqrt{\E R^2}).
			\end{split} 
		\end{equation}
		Combining \eqref{sp}, \eqref{LHS} and \eqref{RHS} yields \eqref{T322a}.
		
		The proof of the lemma is completed. 
	\end{proof}

	Let $P_k(z) = \P(Y_k\le z)= \int_{-\infty}^z p_k(x)\dx x$. For $z\in \R$, let $f_z$ be the solution to the Stein equation 
	\begin{equation}\label{s3} 
		f'(x) - \psi(x)f(x) = \mathbf{1}(x \le z) - P_k(z),
	\end{equation}
	and let $g_z(x)=(\psi(x)f_z(x))'$. In Eichelsbacher and L\"{o}we \cite{eichelsbacher2010stein}, the authors 
	derived some important properties of $f_z$ and $g_z$. These properties will be summarized in the following lemma.
	
	\begin{lemma}[Eichelsbacher and L\"{o}we \cite{eichelsbacher2010stein}]\label{lem.properties.of.solution}
		Let $z\in\R$ and let $P_k$, $f_z$ and $g_z$ be defined as above. The following statements hold.
		\begin{itemize}
			\item[(i)] For $x>0$, we have
			\begin{equation}\label{mills.ratio1}
				1-P_k(x) \le \min \left(\frac{1}{2}, \frac{b_k}{2 k a_k x^{2 k-1}}\right) \exp \left(-a_k x^{2 k}\right),
			\end{equation}
and therefore, for $x<0$, we have
			\begin{equation}\label{mills.ratio2}
				P_k(x)=1-P_k(-x) \le \min \left(\frac{1}{2}, \frac{b_k}{2 k a_k |x|^{2 k-1}}\right) \exp \left(-a_k x^{2 k}\right).
			\end{equation}
			\item[(ii)] The explicit formular for $f_z$ is
			\begin{equation}\label{sol03}
				f_z(x) =
				\begin{cases}
					(1-P_k(z))P_k(x)\exp({a_kx^{2k}})b_k^{-1} &\text{if }\ x < z,\\
					P_k(z)(1-P_k(x))\exp({a_kx^{2k}})b_k^{-1} &\text{if }\ x \ge z.
				\end{cases} 
			\end{equation}
			\item[(iii)] The explicit formular for $f_{z}'$ is
			\begin{equation}\label{sol15}
				f_{z}'(x) =
				\begin{cases}
					(1-P_k(z))\left[1 + 2ka_kx^{2k-1}P_k(x)\exp({a_kx^{2k}})b_k^{-1}\right] &\text{if }\ x < z,\\
					P_k(z)\left[(1-P_k(x))2ka_kx^{2k-1}\exp({a_kx^{2k}})b_k^{-1}-1\right] &\text{if }\ x \ge  z. 
				\end{cases} 
			\end{equation}
			\item[(iv)] 
			For all $x\in\R$, we have
			\begin{equation}\label{sol07}
				0 < f_z(x) \le \dfrac{1}{2b_k}, \text{ and }
				|f_{z}'(x)| \le 1.
			\end{equation}
			\item[(v)] For $x< z$, we have
			\begin{equation*}\label{estimate.lem1.24}
				g_z(x)=(1-P_k(z))\left[b_{k}^{-1}P_k(x)\exp(a_kx^{2k}) \left((4k^2-2k)a_kx^{2k-2} + 4k^2a_k^2x^{4k-2}\right) + 2ka_kx^{2k-1}\right].
			\end{equation*}
		\end{itemize}
	\end{lemma}
	
For any fixed $a>0$, elementary calculations show
	\begin{equation}\label{bound.constant}
	\exp(-a z^{2k})\le \frac{B^pp^{p/k}}{(1+z)^{2p}} \text{ for all } z>0.
	\end{equation}
Hereafter, we will use \eqref{bound.constant} without explicitly mentioning it.
	The following two lemmas are generalizations of Lemmas 2.3 and 2.4 in Th\`{a}nh and Tu \cite{thanh2025non}, respectively.
	\begin{lemma}\label{lem.bs} 
		For $z \ge 5$ and $\xi$ is a random variable satisfying $0\le |\xi|\le |\Delta|$. Then
		\begin{equation}\label{estimate.lem1.01}
			\E(f_{z}(W))^2 \le \dfrac{B^p (p^{p/k}+\E|W|^{2p})}{(1+z)^{2p}},
		\end{equation}
		and
		\begin{equation}\label{estimate.lem1.02}
			\E(f_{z}'(W+\xi))^2 \le \dfrac{B^p (p^{p/k}+\E|W|^{2p})}{(1+z)^{2p}}.
		\end{equation}
	\end{lemma} 
	
	\begin{proof}
		Firstly, we prove \eqref{estimate.lem1.01}.
		To bound $\E(f_{z}(W))^2$, we write 
		\begin{equation}\label{estimate.lem1.03}
			\E (f_{z}(W))^2 =I_1+I_2+I_3,
		\end{equation}
		where
		\begin{equation*}
			\begin{split}
				I_1&=\E (f_{z}(W))^2\mathbf{1}(W< 0),\
				I_2=\E (f_{z}(W))^2\mathbf{1}(W>z/2),\
				I_3=\E (f_{z}(W))^2\mathbf{1}(0\le W\le z/2).
			\end{split}
		\end{equation*}
From  \eqref{mills.ratio2} and \eqref{sol03}, we have 
	\begin{equation}\label{sol05}
			0 < f_z(x) \le B(1-P_k(z)),\ x<0.
		\end{equation}
		By noting that $z\ge 5$ and using \eqref{mills.ratio1} and \eqref{sol05}, we have
		\begin{equation}\label{estimate.lem1.05}
			\begin{split}
				I_1\le B(1-P_k(z))^2 
				\le  B\exp \left(-2a_k z^{2 k}\right)
				\le \dfrac{B^pp^{p/k}}{(1+z)^{2p}}.
			\end{split} 
		\end{equation}
		By using \eqref{sol07} and Markov's inequality, we have
		\begin{equation}\label{estimate.lem1.07}
			\begin{split}
				I_2&\le \left(\frac{1}{2b_k}\right)^2 \P(1+W>1+z/2)\\
				&\le B \frac{\E|1+W|^{2p}}{(1+z/2)^{2p}}\\
				&\le \dfrac{B^p(1+\E|W|^{2p})}{(1+z)^{2p}}.
			\end{split} 
		\end{equation}
		By noting that $z\ge 5$ and using \eqref{mills.ratio1} and \eqref{sol03}, we have
		\begin{equation}\label{estimate.lem1.08}
			\begin{split}
				I_3		&\le b_k^{-2}(1-P_k(z))^2 \E \left(e^{2a_kW^{2k}} \mathbf{1}(0 \le W \le z/2)\right) \\
				&\le B\exp({-2a_kz^{2k}}) \exp({2a_kz^{2k}/4^k})\\
				&=B\exp({-2a_k(1-1/4^k)z^{2k}})\\
				&\le \dfrac{B^p p^{p/k}}{(1+z)^{2p}}.
			\end{split} 
		\end{equation}
		Combining \eqref{estimate.lem1.03}--\eqref{estimate.lem1.08} yields \eqref{estimate.lem1.01}.
		
		Next, we prove \eqref{estimate.lem1.02}. Since $0\le |\xi| \le |\Delta|$,
		\begin{equation}\label{sol12}
			\E|\xi|^{2p} \le \E|W-W'|^{2p} \le 2^{2p-1} (\E|W|^{2p}+\E|W'|^{2p})=2^{2p}\E|W|^{2p}.
		\end{equation}
		To bound $\E(f_{z}'(W+\xi))^2$, we write 
		\begin{equation}\label{estimate.lem1.13}
			\E(f_{z}'(W+\xi))^2 =J_1+J_2+J_3,
		\end{equation}
		where
		\begin{equation*}
			\begin{split}
				J_1&=\E (f_{z}'(W+\xi))^2\mathbf{1}(W+\xi< 0),\\
				J_2&=\E (f_{z}'(W+\xi))^2\mathbf{1}(W+\xi>z/2),\\
				J_3&=\E (f_{z}'(W+\xi))^2\mathbf{1}(0\le W+\xi\le z/2).
			\end{split}
		\end{equation*}
		From  \eqref{mills.ratio2} and \eqref{sol15}, we have  
		\begin{equation}\label{sol13}
			|f_{z}'(x)| \le B(1-P_k(z)),\ x<0.
		\end{equation}
		By using \eqref{mills.ratio1} and \eqref{sol13}, we have
		\begin{equation}\label{estimate.lem1.15}
			\begin{split}
				J_1&\le B(1-P_k(z))^2 \le B\exp \left(-2a_k z^{2 k}\right)\\
				&\le \dfrac{B^pp^{p/k}}{(1+z)^{2p}}.
			\end{split} 
		\end{equation}
		By using \eqref{sol07} and \eqref{sol12}, we have
		\begin{equation}\label{estimate.lem1.17}
			\begin{split}
				J_2&\le \P(W+\xi>z/2)\\
				&\le \frac{\E|W+\xi|^{2p}}{(z/2)^{2p}}\\
				&\le \dfrac{B^p(\E|W|^{2p}+\E|\xi|^{2p})}{(1+z)^{2p}} \\
				&\le \dfrac{B^p(1+\E|W|^{2p})}{(1+z)^{2p}}.
			\end{split} 
		\end{equation}
		By using \eqref{mills.ratio1}  and \eqref{sol15}, we have
		\begin{equation}\label{estimate.lem1.18}
			\begin{split}
				J_3&\le (1-P_k(z))^2\E \left(  ( 1 + 2ka_k(W+\xi)^{2k-1}b_k^{-1}\exp({a_k(W+\xi)^{2k}})  )^2\mathbf{1}(0 < W+\xi \le z/2)\right)\\
				&\le B \exp\left(-2a_kz^{2k}\right)\exp\left(2a_kz^{2k}/4^k\right)\\
				&=B\exp\left(-2a_k(1-1/4^k)z^{2k}\right)\le  \dfrac{B^{p}p^{p/k}}{(1+z)^{2p}}.
			\end{split} 
		\end{equation}
		Combining \eqref{estimate.lem1.13}--\eqref{estimate.lem1.18} yields \eqref{estimate.lem1.02}.
		The proof of the lemma is completed.
	\end{proof}

	\begin{lemma} \label{lem.bsg}  
		Let $z\ge 5$ and $|u| \le 100$. Then
		\begin{equation}\label{estimate.lem1.11}
			\E g_{z}(W+u) \le \dfrac{B^p(p^{p/k}+\E|W|^{2p})}{(1+z)^p}.
		\end{equation}
	\end{lemma} 
	
	\begin{proof}
		Since $\psi(x) = 2ka_kx^{2k-1}$, $|u|\le 100$ and $p\ge 2k-1$, we have from \eqref{sol07} and Lemma \ref{lem.bound.moment.W} that
		\begin{equation}\label{estimate.lem1.21}
			\begin{split}
				\E g_{z}^2(W+u) 
				&\le 2\E\left((\psi'(W+u)f_{z}(W+u))^2 + (\psi(W+u)f_{z}'(W+u))^2\right)\\
				&\le B(\E(W+u)^{4k-4} + \E(W+u)^{4k-2}) \\  
				&\le B^{p}(1+ \E|W|^{2p}).
			\end{split} 
		\end{equation}
		Applying \eqref{estimate.lem1.21}, the Cauchy--Schwarz inequality, Markov's inequality and Lemma \ref{lem.bound.moment.W} again, we obtain
		\begin{equation}\label{estimate.lem1.23}
			\begin{split}
				\E g_{z}(W+u)\mathbf{1}(W+u > z/2)
				&\le (\E g_{z}^2(W+u))^{1/2}(\E\mathbf{1}(W+u > z/2))^{1/2}\\
				&\le \frac{(\E g_{z}^2(W+u))^{1/2}(\E|W+u|^{2p})^{1/2}}{(z/2)^p}\\
				&\le \dfrac{B^{p}(1+\E|W|^{2p})}{(1+z)^{p}}.
			\end{split} 
		\end{equation}
		From Lemma \ref{lem.properties.of.solution} (i) and (v), we have $g(x)$ is increasing on $[0,z)$,
		\begin{equation}\label{sol09}
			0 < g(x) \le B(1-P_k(z))\le \frac{B^pp^{p/k}}{(1+z)^{p}},\ x<0,
		\end{equation}
		and 
		\begin{equation}\label{sol11}
			g(z/2)\le B \exp\left(-2a_kz^{2k}\right)\exp\left(2a_kz^{2k}/4^k\right)\le \frac{B^pp^{p/k}}{(1+z)^{p}}.
		\end{equation}
		It thus follows from \eqref{estimate.lem1.23}--\eqref{sol11} that
		\begin{align*}
			\E g_{z}(W+u)&=\E g_{z}(W+u)\left(\mathbf{1}(W+u<0)+\mathbf{1}(0 \le W+u \le z/2)+\mathbf{1}(W+u > z/2)\right)\\
			&\le \frac{B^{p}p^{p/k}}{(1+z)^{p}} + g_{z}(z/2) + \E(g_{z}(W+u)\mathbf{1}(W+u> z/2))\\
			&\le \frac{B^{p}(p^{p/k}+\E|W|^{2p})}{(1+z)^{p}},
		\end{align*}
		thereby establishing \eqref{estimate.lem1.11}.
\end{proof}

With Lemmas \ref{lem.concentration.inequality}, \ref{lem.bs}, and \ref{lem.bsg} established, the proof of Theorem \ref{thm.main11} 
proceeds along the same lines as the proof of Theorem 1.2 in Th\`{a}nh and Tu \cite{thanh2025non}.  For completeness and to keep the paper self-contained, we provide the details below.
	
\begin{proof}[Proof of Theorem \ref{thm.main11}]
	To  bound $\P(W\le z) - \P(Y_k\le z)$, it suffices to consider $z \ge 0$ since we can simply apply the result to $-W$ when $z<0$. In view of  Theorem \ref{thm.eichelsbacher2010}
	and the fact that
	$\E(\Delta^2I(|\Delta|>a))\le \sqrt{\E(\Delta^4I(|\Delta|>a))},$
	we may assume that $z\ge 5$.
	Since
	\begin{equation*}
		\begin{split}
			|\P(W\le z) - P_k(z)|&=|\P(W> z) - (1-P_k(z))|\\
			&\le \dfrac{\E|1+W|^{2p}}{(1+z)^{2p}}+(1-P_k(z))\\
			&\le \dfrac{C(p)\left(1+\E|W|^{2p}\right)}{(1+|z|)^p},\\
		\end{split}
	\end{equation*}
	the conclusion of the theorem holds if $a\ge 1$. It remains consider the case $0<a<1$.
	
	Let $f_z$ be the solution to Stein's equation \eqref{s3}. Then applying \eqref{sp} again, we have
	\begin{equation} \label{s4}
		\begin{split}
			&\P(W \le z) - P_k(z)\\
			&= \E(f_z'(W) - \psi(W)f_z(W))   \\
			&= \E f_z'(W) - \dfrac{1}{2\lambda} \E(W-W')(f_z(W)-f_z(W')) - \dfrac{1}{\lambda} \E(f_z(W)R)   \\
			&= \E \left[f'_z(W)\left(1- \dfrac{1}{2\lambda}(W - W')^2\right)\right] - \dfrac{1}{\lambda} \E(f_z(W)R) \\
			& \quad  - \dfrac{1}{2\lambda} \E[(W-W')(f_z(W)-f_z(W')- (W-W')f'_z(W))]    \\    
			&:= T_1 + T_2 +T_3.					  	  
	\end{split} \end{equation} 
Using Cauchy--Schwarz inequality and Lemma \ref{lem.bs},  we obtain 
	\begin{align} \label{b01}
		\begin{aligned}
			|T_1| 
			&\le \dfrac{C(p)(1+\E|W|^{2p})}{(1+z)^p }\sqrt{\E\left( 1- \dfrac{1}{2\lambda}\E((W-W')^2|W)\right)^2 },\\
			|T_2|  
			&\le \dfrac{C(p)(1+\E|W|^{2p})}{(1+z)^p}\dfrac{\sqrt{\E(R^2)}}{\lambda}.
	\end{aligned} \end{align}
It remains to bound $T_3$. By recalling $\Delta=W-W'$, we have
	\begin{align} \label{s5}
		\begin{aligned}
			(-2\lambda)T_3 
			&= \E\left(\Delta(f_{z}(W)-f_{z}(W-\Delta)- \Delta f_{z}'(W))\right) \\
			&= \E\left(\Delta \mathbf{1}(|\Delta|>a)(f_{z}(W)-f_{z}(W-\Delta)- \Delta f_{z}'(W))\right)  \\
			&\quad \quad +\E\left(\Delta \mathbf{1}(|\Delta|\le a)(f_{z}(W)-f_{z}(W-\Delta)- \Delta f_{z}'(W))\right)\\
			&:= T_{3,1} + T_{3,2}.
		\end{aligned}
	\end{align}
		By applying Taylor's expansion, the Cauchy--Schwarz inequality and \eqref{estimate.lem1.02}, we obtain,
for a random variable $\xi$ satisfying $0\le |\xi| \le |\Delta|$, that
	\begin{equation}\label{b031} 
		\begin{aligned}
			|T_{3,1}| &= |\E\left(\Delta^2 \mathbf{1}(|\Delta|>a)(f_{z}'(W+\xi)- f_{z}'(W))\right)| \\
			&\le |\E\left(\Delta^2 \mathbf{1}(|\Delta|>a)(f_{z}'(W+\xi)\right)| + |\E\left(\Delta^2 \mathbf{1}(|\Delta|>a) f_{z}'(W))\right)|  \\	
			&\le \left(\E\Delta^4 \mathbf{1}(|\Delta|>a)\E(f_{z}'(W+\xi))^2\right)^{1/2} 
	     		+\left(\E\Delta^4 \mathbf{1}(|\Delta|>a)\E(f_{z}'(W))^2\right)^{1/2} \\
	     	&\le \dfrac{C(p)(1+\E|W|^{2p})\sqrt{\E\Delta^4 \mathbf{1}(|\Delta|>a)}}{(1+z)^p}.
		\end{aligned} 
	\end{equation}
From the Stein equation \eqref{s3}, we have
	\begin{equation}\label{b032}
		\begin{aligned}
			T_{3,2} &=  \E\left(\Delta \mathbf{1}(|\Delta| \le a)\int_{-\Delta}^{0}(f_{z}'(W+t) - f_{z}'(W))\dx t \right)   \\
			&=  \E\left( \Delta \mathbf{1}(|\Delta| \le a)\int_{-\Delta}^{0}(W+t)f_{z}(W+t) - Wf_{z}(W)\dx t \right)   \\
			& \quad +\E\left( \Delta \mathbf{1}(|\Delta| \le a)\int_{-\Delta}^{0}(\mathbf{1}(W+t \le z) - \mathbf{1}(W \le z))\dx t\right)  \\
			&:= T_{3,2,1} + T_{3,2,2}.
		\end{aligned} 
	\end{equation}
	By using Lemma \ref{lem.bsg} and noting that $g_z(w)\ge 0$ for all $w$, we have
	\begin{equation} \label{b0321}
		\begin{aligned}
			|T_{3,2,1}| &= \left|\E\left( \Delta \mathbf{1}(|\Delta| \le a)\int_{-\Delta}^{0} \int_{0}^{t}g_{z}(W+u)\dx u \dx t \right)\right|  \\
			&\le  \left|\E\left( \Delta \mathbf{1}(0 \le \Delta \le a)\int_{-\Delta}^{0}\int_{0}^{t}g_{z}(W+u)\dx u \dx t \right)\right|\\ 
			&\quad+ \left|\E\left( \Delta \mathbf{1}(-a \le \Delta \le 0)\int_{-\Delta}^{0}\int_{0}^{t}g_{z}(W+u)\dx u \dx t \right)\right|\\
			&\le  a\int_{0}^{a}\int_{-t}^{0}\E g_{z}(W+u)\dx u \dx t  + a\int_{0}^{a}\int_{0}^{t}\E g_{z}(W+u)\dx u \dx t \\
			&= 	a \int_{0}^{a}\int_{-t}^{t}\E g_{z}(W+u)\dx u \dx t\\
			& \le  \dfrac{C(p)(1+\E|W|^{2p})a^3}{(1+z)^p}.
		\end{aligned} 
	\end{equation}		
	Similarly,
	\begin{equation} \label{b0324}
		\begin{aligned}
			|T_{3,2,2}|&\le\left|\E\left( \Delta \mathbf{1}(0\le \Delta\le a)\int_{-\Delta}^{0}\mathbf{1}(z< W \le z-t)\dx t\right)\right|\\
			&\quad+\left|\E\left( \Delta \mathbf{1}(-a\le \Delta<0)\int_{-\Delta}^{0}\left(\mathbf{1}(W \le z-t) - \mathbf{1}(W \le z)\right)\dx t \right)\right|\\
			&\le \E\left( \Delta \mathbf{1}(0\le \Delta\le a)\int_{-\Delta}^{0}\mathbf{1}(z-a\le  W \le z+a)\dx t\right) \\
			&\quad+\E\left(-\Delta \mathbf{1}(-a\le \Delta<0)\int_{0}^{-\Delta}\mathbf{1}(z-a\le  W \le z+a)\dx t \right)\\
			&= \E\left( \Delta^2 \mathbf{1}(|\Delta|\le a)\mathbf{1}(z-a\le  W \le z+a)\right)\\
			&\le \frac{C(p)(1+\E|W|^{2p})a(\lambda+\sqrt{\E R^{2}})}{(1+z)^p},
		\end{aligned} 
	\end{equation}
	where we have applied Lemma \ref{lem.concentration.inequality} in the last inequality.
	Combining \eqref{b032}, \eqref{b0321} and \eqref{b0324} yields
	\begin{equation} \label{b0325}
		\begin{aligned}
			|T_{3,2}|&\le \frac{C(p)(1+\E|W|^{2p})(a^3+a(\lambda+\sqrt{\E R^{2}}))}{(1+z)^p}.
		\end{aligned} 
	\end{equation}
	Combining \eqref{s5}, \eqref{b031} and \eqref{b0325} yields
	\begin{align} \label{b03}
		\begin{aligned}
			|T_{3}| &\le \dfrac{C(p)(1+\E|W|^{2p})}{(1+z)^p}\left(\dfrac{a \sqrt{\E R^2}}{\lambda}+ a + \dfrac{ a^3}{\lambda}+\dfrac{\sqrt{\E\Delta^4 \mathbf{1}(|\Delta|>a)}}{\lambda}\right). 
		\end{aligned} 
	\end{align}
	The conclusion of the theorem follows from \eqref{s4}, \eqref{b01} and \eqref{b03}.
	
\end{proof}

	\section{Applications}\label{sec.appl}
	
	In this section, we apply Theorem \ref{thm.main11} to establish non-uniform 
	Berry--Esseen bounds for the magnetization in the critical Curie--Weiss model and
the critical imitative monomer-dimer model.
	Throughout this section, the symbol $B$ denotes a universal constant
	which can be different for each appearance.

	\subsection{The critical Curie--Weiss model}\label{subsec.cw}
	
	Consider the classical Curie--Weiss model
	for $n$ spins $\sigma=(\sigma_1,\ldots,\sigma_n)\in\{-1,1\}^n$. The joint distribution of $\sigma$ is
	\[\P(\sigma)= \dfrac{1}{Z_\beta}
	\exp\left(\dfrac{\beta}{n}\sum_{1\le i<j\le n}\sigma_i\sigma_j \right),
	\]
	where $Z_\beta$ is the normalizing constant and $\beta>0$ is the inverse temperature.
	We are interested in distribution approximations 
	of the total spin (or magnetization) $S_n=\sigma_1+\cdots+\sigma_n$.
	It has been shown that the asymptotic behavior of $S_n$ changes when $\beta$ 
	crosses the critical value $1$. 
	If $0<\beta<1$ or $\beta>1$, then $S_n$ obeys a central limit theorem. However, if $\beta=1$, then $S_n$
	has a non-normal limit. Precisely,
	Ellis and Newman \cite{ellis1978limit}  proved that if $\beta=1$, then $S_n/n^{3/4}$ converges to $Y$
	in distribution, where $Y$ is a random variable with probability density function 
	\begin{equation}\label{ap3}
		p(x)= K\exp(-x^4/12),
	\end{equation}
	and $K$ is the normalizing constant.
	
	Chatterjee and Shao \cite{chatterjee2011nonnormal}, and Eichelsbacher and L\"{o}we \cite{eichelsbacher2010stein} established
	the uniform Berry--Esseen bound for the Ellis--Newman non-central limit theorem
	with rate $n^{-1/2}$ by using Stein's method for exchangeable pairs.
	The Cram\'{e}r-type moderate deviation for $S_n$ was established by
	Can and Pham \cite{can2017cramer}, and Shao et al. \cite{shao2021cramer}
	with the same rate.

	In this subsection, we apply Theorem \ref{thm.main11} to establish a non-uniform 
	Berry--Esseen bound for the total spin $S_n$. Theorem \ref{thm.app1} extends Theorem 2.1
	of Chatterjee and Shao \cite{chatterjee2011nonnormal} and Theorem 3.8 of Eichelsbacher and L\"{o}we \cite{eichelsbacher2010stein}
	to the non-uniform Berry--Esseen bound.
	
	\begin{theorem}\label{thm.app1}
		Let $p\ge 3$. Consider the critical Curie--Weiss model and let $W:=W_n=S_n/n^{3/4}$, where
		$S_n=\sigma_1+\cdots+\sigma_n$ denotes the total spin.
		Then for all $z\in\R$, we have
		\begin{equation}\label{ap2}
			\left|\P\left(W\le z\right) -\P(Y\le z)\right| \le \dfrac{B^{p}p^{p/2}}{(1+|z|)^{p} n^{1/2}},
		\end{equation}
		where $Y$ is a random variable with probability density function as given in \eqref{ap3}.
	\end{theorem}
	
	\begin{proof}
		Let $\sigma' = \{\sigma_{1}',\ldots,\sigma_{n}'\}$, 
		where for each $i$ fixed, $\sigma_{i}'$ is an independent copy of $\sigma_{i}$ given $\{\sigma_{j}, j\ne i\}$, 
		that is, given $\{\sigma_{j}, j\ne i\}$,
		$\sigma_{i}'$ and $\sigma_{i}$ have the same distribution and $\sigma_{i}'$ is conditionally independent of $\sigma_{i}$.
		Let $I$ be a random index independent of all others and uniformly distributed over $\{1,\ldots,n\}$, and let
		\[W'= \frac{S_{n} - \sigma_{I}+ \sigma_{I}'}{n^{3/4}}.\]
		Let $\psi(x)=(x^4/12)'=x^{3}/3$ and $\lambda=1/n^{3/2}$.
		Chatterjee and Shao \cite{chatterjee2011nonnormal} proved that $(W,W')$ is an exchangeable pair satisfying
		\begin{equation}\label{a05}
			\E(W'|W) = W-\lambda \psi(W) +R,
		\end{equation} 
		where $R$ is a random variable.
		We will prove that
		\begin{equation}\label{shao01}
			\E|W|^{2p}\le 20^{p/2} p^{p/2},
		\end{equation}
		\begin{equation}\label{shao02}
			\frac{\sqrt{\E|R|^{2}}}{\lambda}\le \frac{B}{n^{1/2}},
		\end{equation}
		and
		\begin{equation}\label{shao03}
			\E\left( 1- \frac{1}{2\lambda}\E({\Delta}^2|W)\right)^2\le \frac{B}{n},
		\end{equation}
		where $\Delta=W-W'$. By \eqref{a05} and the fact that
		\begin{equation}\label{bounded-pair-CurieWeiss}
			|\Delta|=\left|\frac{\sigma_{I}-\sigma_{I}'}{n^{3/4}}\right| \le \frac{2}{n^{3/4}},
		\end{equation}
		we are now able to apply Theorem \ref{thm.main11} with $k=2$, $p_k(x)=K\exp(-x^4/12)$ and $a=2/n^{3/4}$.
		By collecting the bounds in \eqref{shao01}--\eqref{shao03}, we obtain \eqref{ap2}. 
		
		It remains to prove \eqref{shao01}--\eqref{shao03}. 
		From inequality (5.28) in Chatterjee and Shao \cite{chatterjee2011nonnormal}, we have
		\begin{equation}\label{shao04}
			|R|=\left|\E(W-W'|\sigma)-\frac{W^3}{3n^{3/2}}\right|\le \frac{2|W|^5}{15n^{2}}+\frac{|W|}{n^2}+\frac{1}{n^{11/4}}.
		\end{equation}
		First, assume that $p\ge 3$ is an integer.
		Since $W$ is measurable with respect to $\sigma$, multiplying $|W^{2p-3}|$ on both sides
		of the inequality in \eqref{shao04} and then taking the expectation yields
		\begin{equation}\label{shao05}
			\left|\E(W-W')W^{2p-3}-\frac{\E W^{2p}}{3n^{3/2}}\right|\le \frac{2\E|W|^{2p+2}}{15n^{2}}+\frac{\E|W|^{2p-2}}{n^2}+\frac{\E|W|^{2p-3}}{n^{11/4}},
		\end{equation}
		and therefore
		\begin{equation}\label{shao07}
			\E|W|^{2p}\le 	\left|3n^{3/2}\E(W-W')W^{2p-3}\right|+\frac{2\E|W|^{2p+2}}{5n^{1/2}}+\frac{3\E|W|^{2p-2}}{n^{1/2}}+\frac{3\E|W|^{2p-3}}{n^{5/4}}.
		\end{equation}
		By using the exchangeability of $(W,W')$, we have
\begin{equation}\label{shao09}
			\begin{split}
				\E(W'-W)W^{2p-3}&=\frac{1}{2}\E(W'-W)(W^{2p-3}-(W')^{2p-3})\\
				&=-\frac{1}{2}\E(W'-W)^2(W^{2p-4}+W^{2p-5}W'+\cdots+(W')^{2p-4}).
			\end{split}
\end{equation}
		For any $0\le k\le 2p-4$, we have from H\"{o}lder's inequality that
		\begin{equation}\label{shao10}
			\begin{split}
				\E |W^{2p-4-k}(W')^{k}|\le \left(\E|W|^{2p-4}\right)^{(2p-4-k)/(2p-4)}\left(\E|W'|^{2p-4}\right)^{k/(2p-4)}=\E |W|^{2p-4}.
			\end{split}
		\end{equation} 
Combining \eqref{bounded-pair-CurieWeiss}, \eqref{shao09} and \eqref{shao10} yields
		\begin{equation}\label{shao11}
			\begin{split}
				\left|\E(W'-W)W^{2p-3}\right|\le 2(2p-3)n^{-3/2}\E|W|^{2p-4}.
			\end{split}
		\end{equation}
		Since $|W|\le n^{1/4}$,
		\begin{equation}\label{shao17}
			\frac{2\E|W|^{2p+2}}{5n^{1/2}}+\frac{3\E|W|^{2p-2}}{n^{1/2}}+\frac{3\E|W|^{2p-3}}{n^{5/4}}\le \frac{2\E|W|^{2p}}{5}+3\E|W|^{2p-4}+\frac{3\E|W|^{2p-4}}{n}.
		\end{equation}
		Combining \eqref{shao07}, \eqref{shao11} and \eqref{shao17} yields
		\begin{equation*}
			\E|W|^{2p}\le 	6(2p-3)\E|W|^{2p-4}+\frac{2\E|W|^{2p}}{5}+6\E|W|^{2p-4},
		\end{equation*}
		which is equivalent to
		\begin{equation}\label{shao19}
			\E|W|^{2p}\le 	20(p-1)\E|W|^{2p-4}.
		\end{equation}
		Since 
		\[\E|W|^{2p-4}\le \left(\E|W|^{2p}\right)^{(p-2)/p},\]
		we obtain from \eqref{shao19} that
		\begin{equation}\label{shao21}
			\E|W|^{2p}\le 20^{p/2}(p-1)^{p/2}
		\end{equation}
		thereby establishing \eqref{shao01}. Now, consider the case where $p\ge 3$ is a real number.
		Let $\lceil p \rceil$ denote the smallest integer that is greater or equal to $p$.
		Applying Jensen's inequality and \eqref{shao21} with $p$ replaced by $\lceil p \rceil$, we have
		\begin{equation*}
			\E|W|^{2p}\le \left(\E|W|^{2\lceil p\rceil}\right)^{p/\lceil p\rceil} \le \left(20^{\lceil p\rceil/2}(\lceil p\rceil-1)^{\lceil p\rceil/2} \right)^{p/\lceil p\rceil}\le 20^{p/2} p^{p/2}
		\end{equation*}
		again establishing \eqref{shao01}.

By applying \eqref{shao01} with $p=5$, \eqref{shao04} and Jensen's inequality, we obtain \eqref{shao02}.
		Chatterjee and Shao \cite{chatterjee2011nonnormal} (see the proof of Lemma 5.1 in \cite{chatterjee2011nonnormal}) also proved that 
		\begin{equation}\label{shao23}
			\left|2n^{-3/2}-\E(\Delta^2|\sigma)\right|\le 2n^{-5/2}+2n^{-2}W^2.
		\end{equation}
		Since $\lambda=n^{-3/2}$ and $W$ is measurable with respect to $\sigma$, \eqref{shao23} implies
		\begin{equation}\label{shao24}
			\E\left(1-\frac{1}{2\lambda}\E(\Delta^2|W)\right)^2\le \E\left(n^{-1}+n^{-1/2}W^2\right)^2\le 2(n^{-2}+n^{-1}\E W^4).
		\end{equation}
		By applying \eqref{shao01} with $p=3$, \eqref{shao24}  and Jensen's inequality, we obtain \eqref{shao03}.
		
		The proof of the theorem is completed.
	\end{proof}

	\begin{remark}
\begin{description}
\item[(i)] Th\`{a}nh \cite{thanh2025moment} 
derived a moment inequality for exchangeable pairs, which offers an alternative way to bound $\E|W|^{2p}$.
\item[(ii)] Liu et al. \cite{liu2021non} studied non-uniform Berry--Esseen bound for a general Curie--Weiss model.
		The case $k=2$ of classical Curie--Weiss model of Theorem 4 (ii) of Liu et al. \cite{liu2021non} gives
		\begin{equation}\label{liu21}
			\left|\P\left(W\le z\right) -\P(Y\le z)\right| \le \dfrac{B}{(1+|z|^3) n^{1/4}}.
		\end{equation}
		It would be possible to deal with the general Curie--Weiss model
		using Theorem \ref{thm.main11}. 
		Since this is meant to be only an illustration, we keep the expressions as simple as possible by considering only the classical model.
		Compared to \eqref{liu21}, Theorem \ref{thm.app1} improves the term $(1+|z|^3)$ to
		$(1+|z|^p)$ for any $p\ge3$. In \cite{can2017cramer,chatterjee2011nonnormal,eichelsbacher2010stein}, 
		the rate of convergence for the classical case was already achieved as $n^{-1/2}$.
\end{description}
\end{remark}
	
	\subsection{The imitative monomer-dimer mean-field model}
	
	A mean-field version of the monomer-dimer model with imitative interaction, called the imitative monomer-dimer model, is introduced as follows. 
	For $n \ge 1$, let $G=(V, E)$ be a complete graph with vertex set $V=\{1, \ldots, n\}$ and edge set $E=\{\{u,v\}: u, v \in V, u<v\}$. 
	A dimer configuration on the graph $G$ is a set $D$ of pairwise nonincident 
	edges satisfying the condition that if $\{u,v\} \in D$, then for all $w \neq v$, $(u,w) \notin D$. 
	Let $\mathcal{D}$ denote the set of all dimer configurations. 
	Given a dimer configuration $D$, the set of monomers $\mathcal{M}(D)$ is the collection of dimer-free vertices. 
	The Hamiltonian of the model with an imitation coefficient $J \ge 0$ and an external field $h \in \mathbb{R}$ is given by
	\[-H(D)=n\left(Jm(D)^2+\left(\frac{\log n}{2}+h-J\right) m(D)\right)\]
	for all $D \in \mathcal{D}$, where 
	\[m(D)=|\mathcal{M}(D)|/n.\]
	The associated Gibbs measure is defined as
\[\P(D)=\frac{e^{-H(D)}}{\sum_{D \in \mathcal{D}} e^{-H(D)}}.\]
	Let
	\begin{align} \label{3.7}
		\tilde{p}(x)=-J x^2-\frac{1}{2}(1-g(\tau(x))+\log (1-g(\tau(x)))),
	\end{align}
	where
	\[g(x)=\frac{1}{2}\left(\sqrt{e^{4 x}+4 e^{2 x}}-e^{2 x}\right), \quad \tau(x)=(2 x-1) J+h.\]
	Alberici et al. \cite{alberici2014mean} showed that there exists a function 
	$\gamma:\left(J_c, \infty\right) \rightarrow \mathbb{R}$ with $\gamma\left(J_c\right)=h_c$,
	where $J_c=1/(4(3-2 \sqrt{2}))$ and $h_c=(\log(12-8\sqrt{2})-1)/4$, such that if $(J, h) \notin \Gamma$, where $\Gamma=\left\{(J, \gamma(J)): J>J_c\right\}$, then the function 
	$\tilde{p}(x)$ has a unique maximizer $m_0$ that satisfies $m_0=g\left(\tau\left(m_0\right)\right)$. Moreover, if $(J, h) \notin \Gamma \cup\left\{\left(J_c, h_c\right)\right\}$, then ${\tilde{p}}^{\prime \prime}\left(m_0\right)<0$, and if $(J, h)=\left(J_c, h_c\right)$, then $m_0=m_c=2-\sqrt{2}$, ${\tilde{p}}^{\prime}\left(m_c\right)={\tilde{p}}^{\prime \prime}\left(m_c\right)={\tilde{p}}^{(3)}\left(m_c\right)=0, {\tilde{p}}^{(4)}\left(m_c\right)<0$.
	Alberici et al. \cite{alberici2014mean} proved that if $(J, h) \notin \Gamma \cup\left\{\left(J_c, h_c\right)\right\}$,
	then $m(D)$ concentrates at $m_0$ and has a normal limit. 
	However, if $(J, h)=\left(J_c, h_c\right)$, then
	$m(D)$ still concentrates at $m_0$ and but has a non-normal limit. Specifically, they proved that in this latter case,
	$n^{1 / 4}\left(m(D)-m_0\right)$ converges to $Y$ in distribution, where $Y$ is a random variable with the density function proportional to
	$\exp({\tilde{p}}^{(4)}\left(m_c\right)x^4/24)$.
	
	Chen \cite{chen2016limit} used Stein's method to
	obtain the rate of convergence for the above limit theorems of
	Alberici et al. \cite{alberici2014mean}. To prove his result, Chen \cite{chen2016limit}
	reformulated the monomer-dimer model as a Curie--Weiss model with additional weights.
	Let $n\ge 1$ and $\Sigma=\{0,1\}^n$. For each $\sigma=\left(\sigma_1, \ldots, \sigma_n\right) \in \Sigma$, define a Hamiltonian
	\[-H(\sigma)=n\left(J m(\sigma)^2+\left(\frac{\log n}{2}+h-J\right) m(\sigma)\right),\]
	where 
	\[m(\sigma)=\frac{\sigma_1+\cdots+\sigma_n}{n}\] 
	is the magnetization of the configuration $\sigma$. Let $\mathcal{A}(\sigma)$ denote the set of all sites $i \in V$ such that $\sigma_i=1$, and let $D(\sigma)$ denote the total number of dimer configurations $D \in \mathcal{D}$ with $\mathcal{M}(D)=\mathcal{A}(\sigma)$. By introducing the Gibbs measure
	\[\P(\sigma)=\frac{D(\sigma) \exp (-H(\sigma))}{\sum_{\tau \in \Sigma} D(\tau) \exp (-H(\tau))},\]
	Chen \cite{chen2016limit} showed that
	\[\P\left(m(\sigma)=t/n\right)=\P\left(m(D)=t/n\right)\ \text{ for all }\ t=0,1,\ldots, n.\]
	Therefore, proving the limit theorems for the monomer density in the monomer-dimer model reduces to analyzing the magnetization in the weighted Curie--Weiss model.
	
	Let $\lambda_c=-{\tilde{p}}^{(4)}\left(m_c\right)>0$ and let $Y$ be a random variable with probability density function
	\begin{align}\label{38}
		p(x)=K e^{-\lambda_c x^4 / 24},
	\end{align}
	where $K$ is the normalizing constant. Chen \cite{chen2016limit} proved that if
	$(J, h)=\left(J_c, h_c\right)$, then one has the following uniform Berry--Esseen bound:
	\begin{equation}\label{chen.unif}
		\sup_{z\in\R}\left|\P\left(n^{1/4}(m(\sigma)-m_c)\le z\right)-\P(Y\le z)\right|\le \frac{B}{n^{1/4}}.
	\end{equation}
	More recently, Shao et al. \cite{shao2021cramer} established a Cram\'{e}r type moderate deviation for $m(\sigma)$. 
	In this subsection, we will apply Theorem \ref{thm.main11} to extend \eqref{chen.unif} to a non-uniform Berry--Esseen bound.
	
	\begin{theorem}\label{thm.app2}
		Let $p\ge 3$. Under the above setting, if $(J, h)=\left(J_c, h_c\right)$, then for all $z\in\R$, we have
		\begin{equation}\label{ap4}
			\left|\P\left(n^{1/4}(m(\sigma)-m_c)\le z\right)-\P(Y\le z)\right| \le \dfrac{B^{p}p^{p/2}}{(1+|z|)^{p} n^{1/4}},
		\end{equation}
		where $Y$ is a random variable with probability density function as given in \eqref{38}.
	\end{theorem}
	
	\begin{proof}
		For any $\sigma \in \Sigma, \{u,v\} \in D$ and $s, t \in\{0,1\}$, let $\sigma_{u,v}^{s,t}$ denote the configuration $\tau \in \Sigma$, such that $\tau_i=\sigma_i$ for $i \neq u, v$ and $\tau_u=s, \tau_v=t$. Let ($\sigma_u^{\prime}, \sigma_v^{\prime}$) be independent of ($\sigma_u, \sigma_v$) and follow the conditional distribution
		\[\P\left(\sigma_u^{\prime}=s, \sigma_v^{\prime}=t \mid \sigma\right)=\frac{\P\left(\sigma_{u v}^{s t}\right)}{\sum_{s, t \in\{0,1\}} \P\left(\sigma_{u v}^{s t}\right)}.\]
		Let
		\[W=n^{1/4}(m(\sigma)-m_c)\ \text{ and }\ W^{\prime}=n^{1/4}\left(m(\sigma)-\frac{\sigma_u+\sigma_v-\sigma_u^{\prime}-\sigma_v^{\prime}}{n}-m_c\right).\]
		From Lemma 1 of Chen \cite{chen2016limit} (see also the proof of Theorem 3.2 of Shao et al. \cite{shao2021cramer}), we have
		that $(W,W')$ is an exchangeable pair satisfying
		\begin{equation}\label{a15}
			\E(W'|W) = W-\lambda \psi(W) +R,
		\end{equation} 
		where
		\[\lambda=\dfrac{2(1-m_c)\left(m_{c}^2+(1-m_{c})e^{2\tau(m_{c})}\right)}{\left((1-m_{c})+e^{2\tau(m_{c})}\right)n^{3/2}},\ \psi(x)=\left(\frac{\lambda_c x^4}{24}\right)'=\frac{\lambda_c x^{3}}{6},\]
		and $R$ is a random variable with
		\begin{equation}\label{chen.boundR}
			|R| \le \frac{B\left(|W|^4+1\right)}{n^{7/4}}.
		\end{equation}
		We will prove that
		\begin{equation}\label{chen01}
			\E|W|^{2p}\le B^p p^{p/2},
		\end{equation}
		\begin{equation}\label{chen02}
			\frac{\sqrt{\E|R|^{2}}}{\lambda}\le \frac{B}{n^{1/4}},
		\end{equation}
		and
		\begin{equation}\label{chen03}
			\E\left( 1- \frac{1}{2\lambda}\E({\Delta}^2|W)\right)^2\le \frac{B}{n^{1/2}},
		\end{equation}
		where $\Delta=W-W'$. By \eqref{a15} and the fact that $|\Delta|\le 2/n^{3/4}$,
		we are now able to apply Theorem \ref{thm.main11} with $k=2$, $p_k(x)=K\exp(-\lambda_cx^4/24)$ and $a=2/n^{3/4}$.
		By collecting the bounds in \eqref{chen01}--\eqref{chen03}, we obtain \eqref{ap4}. 
		
		It remains to prove \eqref{chen01}--\eqref{chen03}. 
		From \eqref{a15} and \eqref{chen.boundR}, we have
		\begin{equation}\label{chen04}
			\begin{split}
				\frac{\lambda \lambda_c |W|^3}{6}&=|\E(W-W'|W)+R|\\
				&\le |\E(W-W'|W)|+\frac{B\left(|W|^4+1\right)}{n^{7/4}}.
			\end{split}
		\end{equation}
First, assume that $p\ge 3$ is an integer.
Since $0\le m(\sigma)\le 1$ and $m_c=2-\sqrt{2}$, $|W|\le n^{1/4}$. By using this bound and \eqref{chen04}, we have
		\begin{equation}\label{chen07}
			\begin{split}
				\E|W|^{2p}&\le 	\frac{6}{\lambda\lambda_c}\left(\left|\E(W-W')W^{2p-3}\right|+\frac{B\left(\E|W|^{2p+1}+\E|W|^{2p-3}\right)}{n^{7/4}}\right)\\
				&\le B\left(n^{3/2}\left|\E(W-W')W^{2p-3}\right|+n^{-1/4}\E|W|^{2p+1}+\E|W|^{2p-4}\right).
			\end{split}
		\end{equation}
		As in the proof of \eqref{shao11}, we also have
		\begin{equation}\label{chen11}
			\begin{split}
				\left|\E(W'-W)W^{2p-3}\right|\le 2(2p-3)n^{-3/2}\E|W|^{2p-4}.
			\end{split}
		\end{equation}
		For any $\delta>0$, Chen \cite{chen2016limit} showed that there exist some $\eta>0$ and a constant $B$ such that
		\begin{equation}\label{chen12}
			\begin{split}
				\P(|W|>\delta n^{1/4})\le Be^{-n\eta}.
			\end{split}
		\end{equation}
Since $|W|\le n^{1/4}$, it follows from \eqref{chen12} that
		\begin{equation}\label{chen17}
			\begin{split}
n^{-1/4}\E|W|^{2p+1}&=n^{-1/4}\E|W|^{2p+1}\left(\mathbf{1}(|W|\le \delta n^{1/4})+\mathbf{1}(|W|> \delta n^{1/4})\right)\\
				&\le \delta \E|W|^{2p}+Bn^{p/2}e^{-n\eta}.
			\end{split}
		\end{equation}
		Combining \eqref{chen07}, \eqref{chen11} and \eqref{chen17} yields
		\begin{equation}\label{chen19}
		(1-\delta B)\E|W|^{2p}\le 	B\left((2p-3)\E|W|^{2p-4}+n^{p/2}e^{-n\eta}+\E|W|^{2p-4}\right).
		\end{equation}
By choosing $\delta$ sufficiently small and noting that $n^{p/2}e^{-n\eta}\le B^p p^{p/2}$,
we obtain from \eqref{chen19} that
\begin{equation}\label{chen20}
	\E|W|^{2p}\le B\left(p\E|W|^{2p-4}+B^p p^{p/2}\right).
\end{equation}
Set $u=\E|W|^{2p}$. Then by Jensen's inequality, we have $\E|W|^{2p-4}\le u^{(p-2)/p}$.
Therefore, \eqref{chen20} yields
\begin{equation}\label{chen21}
u\le B\left(pu^{(p-2)/p}+B^p p^{p/2}\right).
\end{equation}
If $u^{(p-2)/p}\le B^p p^{p/2}$, then
\[u\le B^{p^2/(p-2)}p^{p^2/(2(p-2))}\le B^{p+6}p^{(p+6)/2}\]
thereby implying
\begin{equation}\label{chen28}
\E|W|^{2p}\le B^pp^{p/2}.
\end{equation}
If $u^{(p-2)/p}> B^p p^{p/2}$, then it follows from \eqref{chen21} that 
$u\le B(p+1)u^{(p-2)/p}$, which establishes \eqref{chen28} again. Thus, \eqref{chen01} holds when $p\ge 3$ is an integer.
Now, if $p\ge 3$ is a real number, then \eqref{chen01} also follows by applying \eqref{chen28} with $p$ replaced by $\lceil p \rceil$ and using Jensen's inequality.
		
By applying \eqref{chen01} with $p=4$, \eqref{chen.boundR} and Jensen's inequality, we obtain \eqref{chen02}.
		Chen \cite{chen2016limit} (see the proof of Lemma 1 in \cite{chen2016limit}) also proved that 
		\begin{equation}\label{chen29}
			\left|\frac{1}{2 \lambda} \E\left(\Delta^2 \mid W\right)-1\right| \le \frac{B\left(|W|+1\right)}{n^{1/4}}.
		\end{equation}
		By applying \eqref{chen01} with $p=3$, \eqref{chen29}  and Jensen's inequality, we obtain \eqref{chen03}.
		
		The proof of the theorem is completed.
	\end{proof}

	\begin{remark}
		In \cite{thanh2025non}, the normal approximation result (Theorem \ref{thm.main11} for the case $k=1$ and $a_k=1/2$) was applied to derive non-uniform Berry--Esseen bounds for $N$-vector models (see, e.g., \cite{kirkpatrick2016asymptotics,thanh2019error}) and Jack measures (see, e.g., \cite{chen2021error,fulman2004stein,fulman2011zero}). 
It can also be used to obtain non-uniform Berry--Esseen bounds for the Curie--Weiss model and the imitative monomer-dimer model away from the critical point. We leave it to the interested reader.
	\end{remark}
	

 \textbf{Acknowledgments.} The authors are grateful to two anonymous Reviewers for carefully reading the manuscript and for offering
 helpful suggestions to improve the presentation. This work was supported by the National Foundation for Science and Technology Development (NAFOSTED), grant no. 101.03-2021.32.


\end{document}